\documentclass[11pt,dvips,twoside,letterpaper]{article}
\usepackage{pslatex}
\usepackage{fancyhdr}
\usepackage{graphicx}
\usepackage{geometry}
\RequirePackage[latin1]{inputenc} \RequirePackage[T1]{fontenc}

\def\figurename{Figure} 
\makeatletter
\renewcommand{\fnum@figure}[1]{\figurename~\thefigure.}
\makeatother

\def\tablename{Table} 
\makeatletter
\renewcommand{\fnum@table}[1]{\tablename~\thetable.}
\makeatother
\usepackage{color}
                [2000/03/29 v1.0m Standard LaTeX package]

\usepackage{amsmath}
\usepackage{amssymb}
\usepackage{amsfonts}
\usepackage{amsthm,amscd}
\newtheorem{theorem}{Theorem}[section]
\newtheorem{lemma}[theorem]{Lemma}

\newtheorem{proposition}[theorem]{Proposition}
\theoremstyle{definition}

\newtheorem{definition}[theorem]{Definition}

\theoremstyle{remark}
\newtheorem{remark}[theorem]{Remark}

\numberwithin{equation}{section}

\def\P{\mathbb P}
\def\R{\mathbb R}
\def\E{\mathbb E}

\def\Q{\mathbb Q}
\def\X{\mathbb X}

\def\E{\mathbb E}

\def\cal{\mathcal}

\begin{document}

\title{Generalized Backward doubly
SDEs driven by Lévy processes with discontinuous and linear growth coefficients}

\author{Jean Marc Owo\thanks{owo jm@yahoo.fr, corresponding author} \, and Auguste Aman\thanks{augusteaman5@yahoo.fr/aman.auguste@ufhb.edu.ci} \vspace{0.2cm}\\
UFR de Math\'{e}matiques et Informatique\\ $22$ BP $582$ Abidjan, C\^{o}te d'Ivoire
\\ Universit\'{e} de Cocody}

\date{}
\maketitle \thispagestyle{empty} \setcounter{page}{1}

\thispagestyle{fancy} \fancyhead{}
 \fancyfoot{}
\renewcommand{\headrulewidth}{0pt}

\begin{abstract}
This paper deals with generalized backward doubly
stochastic differential equations driven by a Lévy process (GBDSDEL, in short). Under left or right continuous and linear growth conditions, we prove the existence of minimal (resp. maximal) solutions.
\end{abstract}%

\vspace{.08in} \noindent \textbf{MSC}:Primary: 60F05, 60H15; Secondary: 60J30\\
\vspace{.08in} \noindent \textbf{Keywords}: Backward doubly stochastic
differential equations; Lévy processes; Teugels martingales;
comparison theorem; continuous and linear growth conditions.

\section{Introduction}
\text{}

Backward stochastic differential equations (BSDEs in short) are first appeared with the works of J.M. Bismut \cite{Bismut} in linear form 
as the adjoint processes in the maximum principal of stochastic control.
The non-linear BSDEs have been introduced by Pardoux and Peng \cite{PP1}  in order to give a probabilistic interpretation for the solutions of both parabolic and elliptic semi-linear partial differential equations (PDEs) generalizing the well-known Feynman-Kac formula (see Pardoux and Peng \cite{Pg} and Peng \cite{PP2}). Since then, the theory of BSDEs has been developed because of its many applications in the theory of mathematical finance (El Karoui et al. \cite{ELK1}), in stochastic control and stochastic games (El Karoui and Hamadène \cite{KH}, Hamadène and Lepeltier \cite{HL1}, \cite{HL2}).

\smallskip
Furthermore, Pardoux and Zhang \cite{PZ} gave a probabilistic formula for the viscosity solution of a system of PDEs with a nonlinear Neumann boundary condition by introducing a generalized BSDEs (GBSDEs in short) which involved an integral with respect to an adapted continuous increasing process.
\begin{eqnarray}\label{gb}
Y_{t}&=&\xi+\int_{t}^{T}f(s,Y_{s},Z_{s})ds+\int_{t
}^{T}k(s,Y_{s})dA_s-\int_{t}^{T}Z_{s}dW_{s},\,\ 0\leq t\leq
T.
\end{eqnarray}

On the other hand, in order to give a probabilistic representation for a class of quasilinear stochastic partial differential equations (SPDEs in short), Pardoux and Peng \cite{PardPeng} introduced a new class of BSDEs driven by two Brownian motions, the so called backward doubly stochastic differential equations (BDSDEs in short).
Next, using this kind of BDSDEs, 
Bally and Matoussi \cite{BMat} 
gave the probabilistic representation of the weak solutions of parabolic semi linear SPDEs in Sobolev spaces.

\smallskip
Inspired by \cite{PardPeng} and \cite{PZ}, Boufoussi
et al. \cite{Boufsi} recommended a class of generalized BDSDEs (GBDSDEs in short) and gave the probabilistic representation for stochastic viscosity solutions of semilinear SPDEs with a Neumann boundary condition.
\begin{eqnarray}\label{gbd}
Y_{t}&=&\xi+\int_{t}^{T}f(s,Y_{s},Z_{s})ds+\int_{t
}^{T}k(s,Y_{s})dA_s+\int_{t}^{T}g(s,Y_{s},Z_{s})\,\overleftarrow{dB}_{s}\nonumber\\
&&-\int_{t}^{T}Z_{s}dW_{s},\,\ 0\leq t\leq
T.
\end{eqnarray}

In \cite{NS}, Nualart and Schoutens gave a martingale representation theorem associated to Lévy process.
This result allows them to establish in \cite{NSc} the existence and uniqueness result for BSDEs driven by Teugels martingales associated with Lévy process with moments of all orders.

\smallskip Motivated by the above works, especially, by \cite{Boufsi}, \cite{NS} and \cite{NSc}, Hu and Ren \cite{Ren} showed existence and uniqueness result to GBDSDEs driven by Lévy process (GBDSDEL, in short) under Lipschitz condition on the generators and gave the probabilistic interpretation for solutions of a class of stochastic partial differential integral equations (SPDIEs, in short) with a nonlinear Neumann boundary condition.
\begin{eqnarray}\label{inf}
Y_{t}&=&\xi+\int_{t}^{T}f(s,Y_{s^-},Z_{s})ds+\int_{t
}^{T}k(s,Y_{s^-})dA_s+\int_{t}^{T}g(s,Y_{s^-},Z_{s})\,\overleftarrow{d B}_{s}\nonumber\\
&&-\sum_{i=1}^{+\infty}\int_{t}^{T}Z^{(i)}_{s}dH^{(i)}_{s},\,\ 0\leq t\leq
T.
\end{eqnarray}

Recently, Aman and Owo \cite{amo12} relaxed the Lipschitz condition on the generators in Hu and Ren \cite{Ren}.
More precisely, they established under continuous and linear growth conditions on the generators, the existence result to GBDSDEs driven by Lévy process $L$ which have only $m$ different jump size with no continuous part.
\begin{eqnarray}\label{a111}
Y_{t}&=&\xi+\int_{t}^{T}f(s,Y_{s^-},Z_{s})ds+\int_{t
}^{T}k(s,Y_{s^-})dA_s+\int_{t}^{T}g(s,Y_{s^-})\,\overleftarrow{dB}_{s}\nonumber\\
&&-\sum_{i=1}^{m}\int_{t}^{T}Z^{(i)}_{s}dH^{(i)}_{s},\,\ 0\leq t\leq
T.
\end{eqnarray}
The proof is strongly linked to the comparison theorem which does not hold in general for BDSDEs with jumps (see the counter-example in Buckdahn et al. \cite{Barles}). To overcome this difficulty, they assumed an additional relation between the generator $f$ and the jumps size of the Lévy process $L$.

\smallskip Note that, in the previous works on GBDSDEs driven by Lévy process, the generators are at least continuous (see Hu and Ren \cite{Ren} for Lipschitz continuous, and Aman and Owo \cite{amo12} for continuous and linear growth). But, unfortunately, the continuous conditions can not be satisfied in certain models that makes the results in \cite{Ren} and \cite{amo12} not applied for several applications (finance, stochastic control, stochastic games, SPDEs, etc,...).
For example, let the function $f$ be defined by $f(t,y,z)=\displaystyle y1_{\{y>1\}}+\psi(t,z)$, where $\psi$ is a Lipschitz continuous function in $z$.
Such a function $f$ is not continuous in $y$. Then we can not apply the existence results in \cite{Ren} and \cite{amo12} to get the existence of solution of the above GBDSDEL \eqref{a111} with such a coefficient $f$. To correct this shortcoming, we relax in this paper the conditions on the function $f$ in \cite{Ren} and \cite{amo12} by using a left or a right continuous and linear growth conditions and derive the existence of minimal (resp. maximal) solutions to GBDSDEL \eqref{a111}. 

\medskip
\noindent
The paper is organized as follows. In section 2, we give some notations and preliminaries. Section 3 is devoted to the existence of minimal (resp. maximal) solutions result.

\section{Preliminaries}
\subsection{Notations and Definition}
Let $(\Omega, \mathcal{F},\mathbb{P})$ be a complete probability
space on which are defined all the processes stated in this paper and $T$ be
a fixed final time. Throughout this paper, $\{B_{t}; 0\leq t\leq T \}$ will denote a standard one-dimensional
Brownian motion and $\{L_{t}; 0\leq t\leq T \}$ will denote a Lévy process independent of $\{B_{t}; 0\leq t\leq T \}$ corresponding to a standard Lévy measure $\nu$ such that $\int_{\R}(1\wedge x)\nu(dx)<\infty$.
Let $\mathcal{N}$ denote the class of $\P$-null sets of $\mathcal{F}$. For each $t \in [0,T]$, we define
\begin{equation*}
\mathcal{F}_{t}\overset{\Delta}{=}\mathcal{F}_{t}^{L} \vee \mathcal{F}%
_{t,T}^{B},
\end{equation*}
where for any process $\{\eta_{t}\}$ ; $\mathcal{F}_{s,t}^{\eta}=\sigma
\{\eta_{r}-\eta_{s}; s\leq r \leq t \} \vee \mathcal{N}$, $\mathcal{F}%
_{t}^{\eta}=\mathcal{F}_{0,t}^{\eta}$. \newline
Note that $\{\mathcal{F}_{t}^L,\, t\in [0,T]\}$ is an increasing filtration
and $\{\mathcal{F}_{t,T}^B,\, t\in [0,T]\}$ is a decreasing filtration. Thus
the collection $\{\mathcal{F}_{t},\, t\in [0,T]\}$ is neither increasing nor
decreasing so it does not constitute a filtration.

\medskip
\noindent
Furthermore, we denote by $L_{t^{-}}= \underset{s\nearrow t}{\lim}L_{s}$ the left limit process and by $ \Delta L_{t}=  L_{t^{}}-L_{t^{-}}$ the jump size at time $t.$  The power jumps of the Lévy process $L$ are defined by
$$L_{t}^{(1)}=L_{t}\; \text{\; and \;} \; \displaystyle L_{t}^{(i)}=\sum_{0<s\leq t}( \Delta L_{s})^i, \; \; i\geq2.$$
The Teugels Martingale associated with the Lévy process $L$ are defined by $$T_{t}^{(i)}=L_{t}^{(i)}-\E(L_{t}^{(i)})=L_{t}^{(i)}-t\E(L_{1}^{(i)}), \; \; i\geq1$$
Let $(H^{(i)})_{i\geq1}$ be the family of processes defined by
\begin{eqnarray*}
H_{t}^{(i)}=c_{i,i}T_{t}^{(i)}+c_{i,i-1}T_{t}^{(i-1)}+...+c_{i,1}T_{t}^{(1)}, \; \; i\geq1.
\end{eqnarray*}
In \cite{NSc}, Nualart and Schoutens proved  that
the coefficients $c_{i,k}$ correspond to the orthonormalization
of the polynomials $1,\ x, \ x^2,\cdot\cdot\cdot$ with respect to the measure
$\mu(dx)=x^2\nu(dx)+\sigma^2\delta_0(dx)$, i.e $q_{i}(x)=c_{i,i}x^{i-1}+c_{i,i-1}x^{i-2}+...+c_{i,1}$.
The martingale $(H^{(i)})_{i\geq1}$ can be chosen to be
pairwise strongly orthonormal martingale. That is $[H^{(i)},H^{(j)}],\; i\not=j$, and $\left\{[H^{(i)},H^{(i)}]_{t}-t,t\geq0\right\}$ are uniformly integrable martingale with initial value $0$, i.e.,
for all $i,j$,\ $\displaystyle\langle H^{(i)},H^{(j)}\rangle_{t}=\delta_{ij}t$.

\medskip
\noindent
\begin{remark}
 In the case of a Poisson process $N_t$ with intensity $\lambda >0$, all the Teugels martingales are equal to $T_{t}^{(i)}=N_{t}^{}-\lambda t$. Therefore, $H_{t}^{(1)}=\frac{N_{t}^{}-\lambda t}{\sqrt{\lambda}}$ and $H_{t}^{(i)}=0,$\; for all\; $i\geq 2 $.
\end{remark}

\medskip
\noindent
In this paper, we will consider a Lévy process that has only $m$ different jump
sizes with no continuous part, i.e., 
$$H^{(i)}=0,\; \forall\; i\geq m+1 \text{\, and \,}
 \langle H^{(i)},H^{(j)}\rangle_{t}=\delta_{ij}t, \, \, \, \forall t\in [0,T], \, i,j\in\left\{1,...,m\right\}.$$
\medskip
\noindent
In the sequel, $\{A_{t}; 0\leq t\leq T \}$
is a $\mathcal{F}_{t}$-measurable, continuous and increasing real valued process such that $A_0=0$.

\medskip
\noindent
Let us introduce some spaces:\newline
\medskip
\noindent $\bullet$ $\mathcal{M}^{2}(0,T,\mathbb{R}^{m})$ the set of jointly measurable processes $\varphi:\Omega\times[0,T]\rightarrow\R^{m}$, such that $\varphi_t$ is $\mathcal{F}_{t}$-measurable, for any $t \in [0,T],$ with $$\|\varphi \|_{\mathcal{M}^{2}(\R^{m})}^{2}=\displaystyle\mathbb{E}%
\left(\int_{0}^{T}\| \varphi_{t}^{}\|^{2} dt\right)=\displaystyle\sum_{i=1}^{m}\mathbb{E}%
\left(\int_{0}^{T}\mid \varphi_{t}^{(i)} \mid^{2} dt\right)< \infty.$$

\medskip
\noindent $\bullet$ $\mathcal{S}^{2}(0,T)$
the set of jointly measurable continuous processes $\varphi:\Omega\times[0,T]\rightarrow\R$, such that $\varphi_t$ is $\mathcal{F}_{t}$-measurable, for any $t \in [0,T],$ with $\|\varphi \|_{\mathcal{S}^{2}}^{2}=\mathbb{E}\left(%
\underset{0\leq t\leq T}{\sup} \mid \varphi_{t}\mid^{2}\right)< \infty$.

\medskip
\noindent $\bullet$ $\mathcal{A}^{2}(0,T)$
the class of $d\P\otimes dA_t$ a.e. equal measurable random processes $\varphi:\Omega\times[0,T]\rightarrow\R$
such that $\varphi_t$ is $\mathcal{F}_{t}$-measurable, for any $t \in [0,T],$ with$\|\varphi\|_{\mathcal{A}^2}^{2}=\displaystyle\mathbb{E}
\left(\int_{0}^{T}\mid \varphi_{t} \mid^{2} dA_t\right)< \infty.$

\medskip
\noindent The space $\mathcal{E}_{m}(0,T)=\big(\mathcal{S}^{2}(0,T)
 \cap\mathcal{A}^{2}(0,T)\big)\times\mathcal{M}^{2}(0,T,\R^{m})$ endowed with norm
\begin{eqnarray*}
\|(Y,Z)\|_{\mathcal{E}_m}^{2}=\mathbb{E}\left(%
\underset{0\leq t\leq T}{\sup} \mid Y_{t}\mid^{2}
+\int_{0}^{T}\mid Y_{s} \mid^{2} dA_s+
\int_{0}^{T}\|Z_{s}\|^{2} ds\right)
\end{eqnarray*}
is a Banach space.

\begin{definition}
A pair of $\mathbb{R} \times
\R^{m}$-valued process $(Y,Z)$ is called solution of GBDSDEL $(\xi, f,g, h, A)$
driven by Lévy processes if $(Y,Z)\in \mathcal{E}_{m}(0,T)$ and verifies $\eqref{inf}$.
\end{definition}
\begin{definition}
A pair of $\mathbb{R} \times
\R^{m}$-valued process $(\underline{Y},\underline{Z})$ (resp.\,$(\overline{Y},\overline{Z})$) is said to be a minimal (resp.\,maximal) solution of GBDSDEL if for any
other solution $(Y,Z)$ of GBDSDEL, we have $\underline{Y}\leq Y$
(resp. $Y\leq \overline{Y}$).
\end{definition}
\subsection{GBDSDEL with Lipschitz coefficients}
First, we recall the existence and uniqueness result in Ren et al. \cite{Ren} and the comparison theorem
in Aman and Owo \cite{amo12}. To this end, we consider the following assumptions :

\begin{description}
\item {\bf (A1)}\; The terminal value $\xi \in \mathrm{L}^{2}(\Omega,
\mathcal{F}_{T}, \mathbb{P}, \mathbb{R})$ such that for all $\lambda>0$,\ \
$\E (e^{\lambda A_T}|\xi|^2)<\infty.$

\item {\bf (A2)}\; Let $f,g:\Omega \times [0,T]\times \mathbb{R} \times
\mathbb{R}^{m}\rightarrow \mathbb{R}$ and $k:\Omega \times [0,T]\times \mathbb{R} \rightarrow \mathbb{R}$ be jointly measurable processes, such that :

\textbf{(i)}\; for a.e. $t\in[0,T]$ and any $(y,z)\in{\R}^{}\times {\R}^{m},$ $f(t,y,z)$, $g(t,y,z)$ $k(t,y)$ are ${\cal F}_{t }-$measurable;

\textbf{(ii)}\; there exists a constant $K>0$ and three $\mathcal{F}_t$-measurable processes
$\{f_t, g_t, k_t:0\leq t\leq T\}$ with values in $[1,\infty[$, such that for all $t\in[0,T]$, $(y,z)\in{\R}^{}\times {\R}^{m},$\;
$$\left\{
  \begin{array}{ll}
   |f(t,y,z)| \leq f_t +K(|y|+\|z\|) & \hbox{} \\
    |g(t,y,z)| \leq g_t +K(|y|+\|z\|) & \hbox{} \\
    |k(t,y)| \leq k_t +K|y| & \hbox{}
  \end{array}
\right.$$
and for all $\lambda>0$,\; \; $\displaystyle{\E \left(\int_{0}^{T}e^{\lambda A_t}f_t^2dt+\int_{0}^{T}e^{\lambda A_t}g_t^2dt+\int_{0}^{T}e^{\lambda A_t}k_t^2dA_t\right)<\infty};$

\textbf{(iii)}\; there exists two constants $K>0$, $0 < \alpha < 1$ such that for all $t\in[0,T]$, $(y,z), (y^{\prime},z^{\prime})\in{\R}^{}\times {\R}^{m},$
$$\left\{
  \begin{array}{ll}
   |f(t,y,z)-f(t,y^{\prime},z^{\prime})|\leq K(|y-y^{\prime}|+\|z-z^{\prime}\|) & \hbox{} \\
   |g(t,y,z)-g(t,y^{\prime},z^{\prime})|^2\leq K|y-y^{\prime}|^2+\alpha\|z-z^{\prime}\|^2 & \hbox{} \\
   |k(t,y,)-k(t,y')|  \leq K\mid y-y'\mid & \hbox{}
  \end{array}
\right.$$

\end{description}

\begin{theorem}[Ren et al. \cite{Ren}]\label{te1}
Under assumptions $({\bf A1})$ and $({\bf A2})$, the GBDSDEL $\eqref{inf}$ has a unique solution.
\end{theorem}

\noindent Then, considering GBDSDEL $\eqref{a111}$, we have the following comparison theorem established by Aman and Owo in \cite{amo12}.

\begin{theorem}[Aman and Owo \cite{amo12}]\label{tc1}
Assume  $({\bf A1})$ and $({\bf A2})$, and let $(Y^1,Z^1)$ and $(Y^2,Z^2)$ be the solutions of GBDSDEL $\eqref{a111}$ associated with $(\xi^{1},f^{1},g,k^{1},A)$ and $(\xi^{2},f^{2},g,k^{2},A)$ respectively. We suppose:
\begin{itemize}
\item $\xi^2\geq \xi^1, \ \P$-a.s.,
\item $f^2(t,y,z)\geq f^1(t,y,z)$, and $k^2(t,y)\geq k^1(t,y)$ \ $\P$-a.s.,
  for all $(t,y,z)\in [0,T]\times \R\times\R^m$,
\item $\displaystyle\beta^i_t=\frac{f^1(t,y^{2},\widetilde{z}^{(i)})-
 f^1(t,y^{2},\widetilde{z}^{(i-1)})}{z^{2(i)}- z^{1(i)}}
  \mathbf{1}_{\left\{z^{2(i)}\neq z^{1(i)}\right\}}$,
\end{itemize}
where
$\widetilde{z}^{(i)}=\Big(z^{2(1)},z^{2(2)},...,z^{2(i)},z^{1(i+1)},...,z^{1(m)}\Big)$
such that
\begin{eqnarray}
\sum_{i=1}^{m}\beta_t^i\Delta H_t^{(i)}>-1, \ dt\otimes d\P\mbox{-a.s}.\label{adprop}
\end{eqnarray}
Then, we have for all $t\in[0,T],\:\: Y_t^2\geq Y_t^1$, a.s.

\medskip \noindent Moreover, for all $(t,y,z)\in[0,T]\times\mathbb{R} \times \mathbb{R}^m$, if \ $ \xi^{2}> \xi^{1}$, or $f^{2}(t,y,z)> f^{1}(t,y,z)$, \\ or $k^{2}(t,y)> k^{1}(t,y)$, a.s.,
$Y_{t}^{2}> Y_{t}^{1},\ \ a.s.,\ \forall\, t \in [0,T]$.
\end{theorem}

\begin{remark}
- If the Lévy process $L$ has only positive jumps size, then the inequality \eqref{adprop} holds when, for each $(t,y)\in [0,T]\times \R$, the function $z\mapsto f(t,y,z)$ is non-decreasing.

\noindent
- If the Lévy process $L$ has only negative jumps size, then the inequality \eqref{adprop} holds when, for each $(t,y)\in [0,T]\times \R$, the function $z\mapsto f(t,y,z)$ is non-increasing.
\end{remark}
\begin{remark}
If the inequality \eqref{adprop} does not hold, then, in general, the comparison theorem does not hold. We give now a
counter-example.
\end{remark}
\medskip \noindent Let $N_t$ be a Poisson process with intensity $\lambda >0$, then $H_{t}^{(1)}=\frac{N_{t}^{}-\lambda t}{\sqrt{\lambda}}$, and $H_{t}^{(i)}=0,$ \;$i\geq 2 $. \\ Now, let $f(t,y,z)=-(1+\sqrt{\lambda})z$,\; and consider the following equations $(j=1,2)$ :
\begin{eqnarray*}
Y_{t}^{j}&=&\xi^{j}+\int_{t}^{T}-(1+\sqrt{\lambda})Z_{s}^{j}ds-\int_{t}^{T}Z^{j}_{s}dH^{(1)}_{s},\,\ 0\leq t\leq
T.\label{}
\end{eqnarray*}
If we choose $$ \xi^{2}=\frac{N_{T}}{\sqrt{\lambda}} \; \; \text{and} \; \; \xi^{1}=0,$$
then
$$ (Y_t^{2},Z_t^{2})=(\frac{N_{t}}{\sqrt{\lambda}}+t-T,1)\; \; \text{and} \; \; (Y_t^{1},Z_t^{1})=(0,0).$$
Clearly, $ \xi^{2}\geq\xi^{1},$ but $\P(Y_t^{2}<Y_t^{1})= \P(N_{t}<\sqrt{\lambda}(T-t))>0,$ for all $0\leq t<T$.

\subsection{GBDSDEL with continuous coefficients}
Furthermore, using the above comparison theorem and the well know approximation method of the functions $f$ and $k$, Aman and Owo \cite{amo12} proved the existence of a minimal or maximal solution for GBDSDEL $\eqref{a111}$ under the continuous and linear growth conditions.  In this section, we recall this result. In this fact, we consider the following assumptions :
\begin{description}
\item {\bf (H1)}\; The terminal value $\xi \in \mathrm{L}^{2}(\Omega,
\mathcal{F}_{T}, \mathbb{P}, \mathbb{R})$ such that for all $\lambda>0$,\ \
$\E (e^{\lambda A_T}|\xi|^2)<\infty.$

\item {\bf (H2)}\; Let $f:\Omega \times [0,T]\times \mathbb{R} \times
\mathbb{R}^{m}\rightarrow \mathbb{R}$ and $g,k:\Omega \times [0,T]\times \mathbb{R} \rightarrow \mathbb{R}$ be jointly measurable processes, such that :

\textbf{(i)}\; for a.e. $t\in[0,T]$ and any $(y,z)\in{\R}^{}\times {\R}^{m},$ $f(t,y,z)$, $g(t,y)$ $k(t,y)$ are ${\cal F}_{t }-$measurable;

\textbf{(ii)}\; there exists a constant $K>0$ and three $\mathcal{F}_t$-measurable processes
$\{f_t,g_t, k_t:0\leq t\leq T\}$ with values in $[1,\infty[$, such that for all $t\in[0,T]$, $(y,z)\in{\R}^{}\times {\R}^{m},$\;
$$\left\{
  \begin{array}{ll}
   |f(t,y,z)| \leq f_t +K(|y|+\|z\|) & \hbox{} \\
 |g(t,y)| \leq g_t +K|y|,\;  \hbox{with\; } g(t,0)=0&  \\
    |k(t,y)| \leq k_t +K|y| & \hbox{}
  \end{array}
\right.$$
\textbf{(iii)}\; for all $\mu, \lambda>0$,\; \; $\displaystyle{\E \left(\int_{0}^{T}e^{\mu t+\lambda A_t}f_t^2dt+\int_{0}^{T}e^{\mu t+\lambda A_t}g_t^2dt+
\int_{0}^{T}e^{\mu t+\lambda A_t}k_t^2dA_t\right)<\infty};$

\textbf{(iv)}\; there exists a constant $K>0$, such that for all $t\in[0,T]$, $(y,z), (y^{\prime},z^{\prime})\in{\R}^{}\times {\R}^{m},$
$$\left\{
  \begin{array}{ll}
   |f(t,y,z)-f(t,y,z^{\prime})|\leq K\|z-z^{\prime}\| & \hbox{} \\
   |g(t,y)-g(t,y^{\prime})|^2\leq K|y-y^{\prime}|^2 & \hbox{}
  \end{array}
\right.$$
\item [(v)] for all $(t,z)\in[0,T]\times {\R}^{m}$, $y\mapsto f(t,y,z)$ and $y\mapsto k(t,y)$ are continuous, for all $\omega$ a.e.
\end{description}

\begin{theorem}[Aman and Owo \cite{amo12}]\label{te}
Under assumptions $({\bf H1})$ and $({\bf H2})$, the GBDSDEL \eqref{a111}
has a minimal solution $(\underline{Y},\underline{Z})\in \mathcal{E}_{m}(0,T)$ (resp. a maximal solution $(\overline{Y},\overline{Z})\in \mathcal{E}_{m}(0,T)$).
\end{theorem}
Moreover, using the same argument as in \cite{amo12}, one can establish the following comparison theorem which requires at least one Lipschitz function :
\begin{theorem}[Conparison with at least one Lipschitz function]\label{l0a1}
Let $g$ and $\xi^i$ $(i=1,2)$ satisfy ${\bf (H1)}$ and ${\bf(H2)}$. Assume that GBDSDEL $(\xi^{i}, f^{i},g, k^{i}, A)$  have solutions $(Y^{i},Z^{i})\in {\cal E}_{m} (0,T )$, $i=1,2$, respectively. Assume moreover that
\begin{itemize}
\item $\xi^2\geq \xi^1, \ \P$-a.s.,
\item $f^1$ and $k^{1}$ satisfy ${\bf (A2)}$ such that $f^2(t,Y_{t}^2,Z_{t}^2)\geq f^1(t,Y_{t}^2,Z_{t}^2)$, and $k^2(t,Y_{t}^2)\geq k^1(t,Y_{t}^2)$ \ $\P$-a.s., for all $t\in [0,T]$, with
$$\displaystyle\beta^i_t=\frac{f^1(t,Y_{t}^{2},\widetilde{Z}_{t}^{(i)})-
 f^1(t,Y_{t}^{2},\widetilde{Z}_{t}^{(i-1)})}{Z_{t}^{2(i)}- Z_{t}^{1(i)}}
  \mathbf{1}_{\left\{Z_{t}^{2(i)}\neq Z_{t}^{1(i)}\right\}},$$
where
$\widetilde{Z}_{t}^{(i)}=\Big(Z_{t}^{2(1)},Z_{t}^{2(2)},...,Z_{t}^{2(i)},Z_{t}^{1(i+1)},...,Z_{t}^{1(m)}\Big)$
\item[]\text{} \Big(resp. $f^2$ and $k^{2}$ satisfy ${\bf(A2)}$ such that $f^2(t,Y_{t}^1,Z_{t}^1)\geq f^1(t,Y_{t}^1,Z_{t}^1)$, and $k^2(t,Y_{t}^1)\geq k^1(t,Y_{t}^1)$ \ $\P$-a.s., for all $t\in [0,T]$,  with
$$\displaystyle\beta^i_t=\frac{f^2(t,Y_{t}^{1},\widetilde{Z}_{t}^{(i)})-
 f^2(t,Y_{t}^{1},\widetilde{Z}_{t}^{(i-1)})}{Z_{t}^{1(i)}- Z_{t}^{2(i)}}
  \mathbf{1}_{\left\{Z_{t}^{1(i)}\neq Z_{t}^{2(i)}\right\}},$$
\end{itemize}
where
$\widetilde{Z}_{t}^{(i)}=\Big(Z_{t}^{1(1)},Z_{t}^{1(2)},...,Z_{t}^{1(i)},Z_{t}^{2(i+1)},...,Z_{t}^{2(m)}\Big)$\Big),
such that
\begin{eqnarray}
\sum_{i=1}^{m}\beta_t^i\Delta H_t^{(i)}>-1, \ dt\otimes d\P\mbox{-a.s}.\label{ad1}
\end{eqnarray}

\medskip\noindent  Then, $Y_t^2\geq Y_t^1$\ \ a.s., for all $t\in[0,T]$.
\end{theorem}

\section{GBDSDEL with left-continuous coefficients}
The objective of this section is to prove an existence of minimal solution to GBDSDEL $(\xi, f,g, k, A)$ with left-continuous and linear growth coefficients. More precisely, when the following assumptions hold.

\begin{description}
\item {\bf (C1)}\; The terminal value $\xi \in \mathrm{L}^{2}(\Omega,
\mathcal{F}_{T}, \mathbb{P}, \mathbb{R})$ such that for all $\lambda>0$,\ \
$\E (e^{\lambda A_T}|\xi|^2)<\infty.$

\item {\bf (C2)}\; Let $f:\Omega \times [0,T]\times \mathbb{R} \times
\mathbb{R}^{m}\rightarrow \mathbb{R}$ and $g,k:\Omega \times [0,T]\times \mathbb{R} \rightarrow \mathbb{R}$ be jointly measurable processes, such that :

\textbf{(i)}\; for a.e. $t\in[0,T]$ and any $(y,z)\in{\R}^{}\times {\R}^{m},$ $f(t,y,z)$, $g(t,y)$ $k(t,y)$ are ${\cal F}_{t }-$measurable;

\textbf{(ii)}\; there exists a constant $K>0$ and three $\mathcal{F}_t$-measurable processes
$\{f_t,g_t, k_t:0\leq t\leq T\}$ with values in $[1,\infty[$, such that for all $t\in[0,T]$, $(y,z)\in{\R}^{}\times {\R}^{m},$\;
$$\left\{
  \begin{array}{ll}
   |f(t,y,z)| \leq f_t +K(|y|+\|z\|) & \hbox{} \\
 |g(t,y)| \leq g_t +K|y|,\;  \hbox{with\; } g(t,0)=0& \\
    |k(t,y)| \leq k_t +K|y| & \hbox{}
  \end{array}
\right.$$
\textbf{(iii)}\; for all $\mu, \lambda>0$,\; \; $\displaystyle{\E \left(\int_{0}^{T}e^{\mu t+\lambda A_t}f_t^2dt+\int_{0}^{T}e^{\mu t+\lambda A_t}g_t^2dt+
\int_{0}^{T}e^{\mu t+\lambda A_t}k_t^2dA_t\right)<\infty};$

\textbf{(iv)}\; there exists a constant $K>0$, such that for all $t\in[0,T]$, $(y,z), (y^{\prime},z^{\prime})\in{\R}^{}\times {\R}^{m},$
$$\left\{
  \begin{array}{ll}
   |f(t,y,z)-f(t,y,z^{\prime})|\leq K\|z-z^{\prime}\| & \hbox{} \\
   |g(t,y)-g(t,y^{\prime})|^2\leq K|y-y^{\prime}|^2 & \hbox{}
  \end{array}
\right.$$
\textbf{(v)}\; for all $(t,z)\in[0,T]\times {\R}^{m}$, $y\mapsto f(t,y,z)$ and $y\mapsto k(t,y)$ are left-continuous, for all $\omega$ a.e.

\textbf{(vi)}\; there exists two functions
$h:\mathbb{R} \times \mathbb{R}^{m}\rightarrow \mathbb{R}$ and $p:\mathbb{R} \rightarrow \mathbb{R}$, satisfying, for all $y\in \mathbb{R}$, $z, z^{\prime} \in \mathbb{R}^{
 m}$:
$$\left\{
  \begin{array}{ll}
 |h(y,z)| \leq K(|y|+\|z\|) \hbox{\, and\,}   |p(y)| \leq K|y|, & \hbox{} \\
 |h(y,z)-h(y,z^{\prime})|\leq K\|z-z^{\prime}\|, & \hbox{} \\
 y\mapsto h(y,z)\hbox{\, and\, } y\mapsto p(y) \; \hbox{are continuous} & \hbox{}
  \end{array}
\right.$$
such that for all $y\geq y^{\prime}$, $z, z^{\prime} \in \mathbb{R}^{
 m}$, $t\in[0,T]$, we have
$$f(t,y,z)-f(t,y^{\prime},z^{\prime}) \geq
h(y-y^{\prime},z-z^{\prime})\hbox{\, and\, } k(t,y)-k(t,y^{\prime}) \geq
p(y-y^{\prime}).$$
\end{description}

In order to reach our objective, we need some preliminary results. 

\smallskip
First, we prove the following comparison theorem which requires at least one continuous function.
\begin{theorem}[Conparison with at least one continuous function]\label{l0a2}
Let $g$ and $\xi^i$ $(i=1,2)$ satisfy ${\bf (C1)}$ and ${\bf (C2)}$. Assume that GBDSDEL $(\xi^{i}, f^{i},g, k^{i}, A)$  have solutions $(Y^{i},Z^{i})\in {\cal E}_{m} (0,T )$, $i=1,2$, respectively. Assume moreover that

\begin{itemize}
\item $\xi^2\geq \xi^1, \ \P$-a.s.
\item $f^1$ and $k^{1}$ satisfy ${\bf(H2)}$ such that $f^2(t,Y_{t}^2,Z_{t}^2)\geq f^1(t,Y_{t}^2,Z_{t}^2)$, and $k^2(t,Y_{t}^2)\geq k^1(t,Y_{t}^2)$ \ $\P$-a.s., for all $t\in [0,T]$, and $(Y^{1},Z^{1})$ is the minimal solution, with
$$\displaystyle\beta^i_t=\frac{f^1(t,Y_{t}^{2},\widetilde{Z}_{t}^{(i)})-
 f^1(t,Y_{t}^{2},\widetilde{Z}_{t}^{(i-1)})}{Z_{t}^{2(i)}- Z_{t}^{1(i)}}
  \mathbf{1}_{\left\{Z_{t}^{2(i)}\neq Z_{t}^{1(i)}\right\}},$$
where
$\widetilde{Z}_{t}^{(i)}=\Big(Z_{t}^{2(1)},Z_{t}^{2(2)},...,Z_{t}^{2(i)},Z_{t}^{1(i+1)},...,Z_{t}^{1(m)}\Big)$
\item[]\text{} \Big(resp. $f^2$ and $k^{2}$ satisfy ${\bf(H2)}$ such that $f^2(t,Y_{t}^1,Z_{t}^1)\geq f^1(t,Y_{t}^1,Z_{t}^1)$, and $k^2(t,Y_{t}^1)\geq k^1(t,Y_{t}^1)$ \ $\P$-a.s., for all $t\in [0,T]$, and $(Y^{2},Z^{2})$ is the maximal solution, with
$$\displaystyle\beta^i_t=\frac{f^2(t,Y_{t}^{1},\widetilde{Z}_{t}^{(i)})-
 f^2(t,Y_{t}^{1},\widetilde{Z}_{t}^{(i-1)})}{Z_{t}^{1(i)}- Z_{t}^{2(i)}}
  \mathbf{1}_{\left\{Z_{t}^{1(i)}\neq Z_{t}^{2(i)}\right\}},$$
\end{itemize}
where
$\widetilde{Z}_{t}^{(i)}=\Big(Z_{t}^{1(1)},Z_{t}^{1(2)},...,Z_{t}^{1(i)},Z_{t}^{2(i+1)},...,Z_{t}^{2(m)}\Big)$\Big),
such that
\begin{eqnarray}
\sum_{i=1}^{m}\beta_t^i\Delta H_t^{(i)}>-1, \ dt\otimes d\P\mbox{-a.s}.\label{ad2}
\end{eqnarray}

\medskip\noindent  Then, $Y_t^2\geq Y_t^1$\ \ a.s., for all $t\in[0,T]$.
\end{theorem}
\begin{proof}
\noindent Assume that $f^1$ and $k^{1}$ satisfy ${\bf(H2)}$. Then, 
by Lemma 3.2 in  \cite{amo12}, there exist sequences of functions $f_n^{1}$ and $k_n^{1}$ associated to $f^{1}$ and $k^{1}$, respectively, defined by :
$$f_{n}^{1}(t,y,z)=\underset{u\in \Q}\inf \left\{f^{1}(t,u,z)+n|y-u|\right\} \; \; \text{and} \; \; k_{n}^{1}(t,y)=\underset{u\in \Q}\inf \left\{k^{1}(t,u)+n|y-u|\right\} $$
which are Lipschitz functions such that $f_{n}^{1} \leq f^{1}$ and $k_{n}^{1} \leq k^{1}$, for all $n\geq K$.
Furthermore, $f_n^{1}$ and $k_n^{1}$ converge suitably to $f^{1}$ and $k^{1}$, respectively. Hence, by Theorem \ref{te1}, it follows that, for every $n\geq K$, GBDSDEL $(\xi^{1}, f_{n}^{1},g, k_{n}^{1}, A)$ has a unique measurable solution $(Y^{1,n}, Z^{1,n})$. Moreover, from the proof of the Theorem 3.1 in  \cite{amo12}, it follows that $(Y^{1,n}, Z^{1,n})$ converges to the minimal solution $(Y^1,Z^1)$ of GBDSDEL $(\xi^{1}, f^{1},g, k^{1}, A)$.
On the other hand, since $f^1(t,Y^2,Z^2)\leq f^2(t,Y^2,Z^2)$ and $k^1(t,Y^2)\leq k^2(t,Y^2)$ \ a.s., we have  $$f_n^1(t,Y^2,Z^2)\leq f^2(t,Y^2,Z^2)\; \; and \; \;k_n^1(t,Y^2)\leq k^2(t,Y^2),\; \; a.s.,\; \; \text{for all}\; \; n\geq K.$$
Moreover, by the definition of $f_{n}$, it follows that
$$f_{n}^1(t,Y_{t}^{2},\widetilde{Z}_{t}^{(i)})-
 f_{n}^1(t,Y_{t}^{2},\widetilde{Z}_{t}^{(i-1)})= f^1(t,Y_{t}^{2},\widetilde{Z}_{t}^{(i)})-
 f^1(t,Y_{t}^{2},\widetilde{Z}_{t}^{(i-1)}),$$
and inequality \eqref{ad2} still holds.
\noindent Therefore, by Theorem \ref{l0a1}, we get  $Y^{1,n}\leq Y^2$ a.s., for all $n\geq K$. Hence, we have $Y^{1}\leq Y^2$.
\end{proof}
Next, we establish the following result, which will be useful in the sequel.
\begin{lemma}\label{l1}
Let $g$, $\xi$ satisfy ${\bf(C1)}$ and ${\bf(C2)}$, and $h$ and $p$ be the functions appear in assumption ${\bf(C2)}$.
Assume moreover that, for any $y\in\mathbb{R}$, the rate of change of $h(y,\cdot)$ :
$$\displaystyle\alpha^i=\frac{h(y^{},\widetilde{z}^{(i)})-
 h(y^{},\widetilde{z}^{(i-1)})}{z^{2(i)}- z^{1(i)}}
  \mathbf{1}_{\left\{z^{2(i)}\neq z^{1(i)}\right\}},$$
where 
$\widetilde{z}^{(i)}=\Big(z^{2(1)},z^{2(2)},...,z^{2(i)},z^{1(i+1)},...,z^{1(m)}\Big)$, for all $z^1,z^2\in \R^m$
satisfies the following relation :
\begin{eqnarray}
\sum_{i=1}^{m}\alpha^i\Delta H_t^{(i)}>-1, \ dt\otimes d\P\mbox{-a.s}.\label{ad3}
\end{eqnarray}
Let $\phi$ and $\psi$ be processes such that, for a.e. $t\in[0,T]$, $\phi(t)$ and $\psi(t)$ are ${\cal F}_{t}-$measurable and ${\E}\left(\displaystyle\int_{0}^{T }\left|
\phi(s)\right| ^{2}ds+\int_{0}^{T }\left|
\psi(s)\right| ^{2}dA_s\right) <+\infty$.
We consider the following GBDSDEL :
\begin{eqnarray}
Y_{t}&=&\xi+\int_{t}^{T}\left(h(Y_{s^{-}},Z_s)+\phi(s)\right)ds+\int_{t
}^{T}\left(p(Y_{s^{-}})+\psi(s)\right)dA_s+\int_{t}^{T}g(s,Y_{s^-})\,d\overleftarrow{B}_{s}\nonumber\\
&&-\sum_{i=1}^{m}\int_{t}^{T}Z^{(i)}_{s}dH^{(i)}_{s},\,\ 0\leq t\leq
T.\label{2}
\end{eqnarray}
Then, the GBDSDEL \eqref{2} has a solution $(Y,Z) \in{\cal E}_{m}\left(0,T \right)$. Moreover, for any solution $(Y,Z)$ of GBDSDEL \eqref{2}, if $\xi \geq 0$, and
$\phi(t)\geq0$ and $\psi(t)\geq0$ for any $t\in [0,T]$, we have $Y_t\geq 0,$ \ $\P-$a.s. for all $t\in [0,T]$.
\end{lemma}
\begin{proof}
Since $h$ and $p$ are continuous and linear growth, by virtue of Theorem \ref{te}, the GBDSDEL \eqref{2} has at least one solution $(Y,Z) \in{\cal E}\left(0,T \right)$. 
Moreover by Lemma 3.2 in  \cite{amo12}, there exist non-decreasing sequences of Lipschitz functions $h_n^{}$ and $p_n^{}$ associated to $h$ and $p$, respectively, defined by:
$$h_{n}^{}(y,z)=\underset{u\in \Q}\inf \left\{h(u,z)+n|y-u|\right\} \; \; \text{and} \; \; p_{n}(y)=\underset{u\in \Q}\inf \left\{p(u)+n|y-u|\right\} ,$$
with $|h_{n}^{}(y,z)|\leq K(|y|+|z|)$ and $|p_{n}(y)|\leq K|y|$. 

\noindent In view of Theorem \ref{te1}, 
for every $n\geq K$, the equations
\begin{eqnarray}
Y_{t}^{n}&=&\xi+\int_{t}^{T}\left(h_{n}(Y_{s^{-}}^{n},Z_s^{n})+\phi(s)\right)ds+\int_{t
}^{T}\left(p_{n}(Y_{s^{-}}^{n})+\psi(s)\right)dA_s+\int_{t}^{T}g(s,Y_{s^-}^{n})\,d\overleftarrow{B}_{s}\nonumber\\
&&-\sum_{i=1}^{m}\int_{t}^{T}Z^{n(i)}_{s}dH^{(i)}_{s},\label{s}
\end{eqnarray}
and
\begin{eqnarray}
y_{t}^{n}=\int_{t}^{T}h_{n}(y_{s^{-}}^{n},z_s^{n})ds+\int_{t
}^{T}p_{n}(y_{s^{-}}^{n})dA_s+\int_{t}^{T}g(s,y_{s^-}^{n})\,d\overleftarrow{B}_{s}
-\sum_{i=1}^{m}\int_{t}^{T}z^{n(i)}_{s}dH^{(i)}_{s}\label{t}
\end{eqnarray}
have unique solutions $(Y^n,Z^n)$ and $(y^n,z^n)$ in $\mathcal{E}(0,T)$, respectively.

\noindent Since $g(s,0)=h_n(0,0)=p_n(0)=0$, the solution $(y^n,z^n)$ of \eqref{t} is $(y^n,z^n)=(0,0)$.

\noindent Now, since  $\xi \geq 0$, $\phi(t)\geq0$, $\psi(t)\geq0$, for any $t\in [0,T]$ and inequality \eqref{adprop} still holds for $h_n$ with any $n\geq K$, it follows from Theorem \ref{tc1} that, for any $n\geq K$, $Y^{n}\geq y^n=0$, a.s..

\noindent Moreover, the solution $(Y^n,Z^n)$ of \eqref{s} converges to the minimal solution $(\underline{Y},\underline{Z})$ of \eqref{2} (see the proof of the Theorem 3.1 in  \cite{amo12}). Therefore, $\underline{Y}\geq0$, a.s.. Consequently, for any solution $(Y,Z)$ of GBDSDEL \eqref{2}, we have $Y\geq0.$
\end{proof}
\begin{remark}\label{rk}
Lemma \ref{l1} still holds if $h$ and $p$ depend, as well as, on
$t\in[0,T]$ with $h(t,0,0)\geq0$ and $p(t,0)\geq0$, for all $t\in[0,T]$, such that
for a.e. $t\in[0,T]$, $h(t,0,0)$ and $p(t,0)$ are ${\cal F}_{t}-$measurable
and ${\E}\left(\displaystyle\int_{0}^{T }\left|
h(t,0,0)\right| ^{2}ds+\int_{0}^{T }\left|
p(t,0)\right| ^{2}dA_s\right) <+\infty$, and $|h(t,y,z)-h(t,0,0)| \leq K(|y|+\|z\|)$ and $|p(t,y)-p(t,0)| \leq K|y|$. Indeed, it suffices to replace $(h,\phi)$ and $(p,\psi)$ in Lemma \ref{l1} by :
$$\left\{
  \begin{array}{ll}
    \tilde{h}(y,z) = h(t,y,z) -h(t,0,0) \hbox{\; and \;} \tilde{\phi}(t)= \phi(t)+h(t,0,0),& \\
\\
    \tilde{p}(t,y) = p(t,y) -p(t,0) \hbox{ \; and \;} \tilde{\psi}(t)= \psi(t)+p(t,0).&
  \end{array}
\right.$$
\end{remark}

\smallskip
\noindent Now, we are ready to prove the existence of solutions for GBDSDEL \eqref{a111}.
\begin{theorem}\label{te3}
Under the assumptions $({\bf C1})$ and $({\bf C2})$, the GBDSDEL \eqref{a111} has a minimal solution $(\underline{Y},\underline{Z}) \in{\cal E}_{m}\left(0,T \right)$. 
\end{theorem}
\begin{proof}
Let us construct an approximate sequence using a Picard-type iteration, by virtue of the results on GBDSDEL with continuous and linear growth coefficients, due to \cite{amo12}.

\smallskip\noindent
First, let $(\underline{Y}^0,\underline{Z}^0)\in {\cal E}_{m}\left(0,T \right)$ be the unique solution of
\begin{eqnarray}\label{11a}
\underline{Y}_t^{0}&=&\xi+\int_{t}^{T}\left(-f_s-K|\underline{Y}_s^{0}|-K\|\underline{Z}_s^{0}\|\right)ds
+\int_{t}^{T}\left(-k_s-K|\underline{Y}_s^{0}|\right)dA_s
+\int_{t}^{T}g(s,\underline{Y}_s^{0})\overleftarrow{dB}_{s}\notag\\ &&-
\sum_{i=1}^{m}\int_{t}^{T}\underline{Z}^{0(i)}_{s}dH^{(i)}_{s}.
\end{eqnarray}
Now let $\left\{(\underline{Y}^n,\underline{Z}^n)\right\}_{n\geq0}$ be a sequence in ${\cal E}_{m}\left(0,T \right)$ defined recursively by $(\underline{Y}^0,\underline{Z}^0)$ the solution of \eqref{11a} and for any $n\geq 1$, 
\begin{eqnarray}\label{11}
\underline{Y}_t^n&=&\xi+\int_{t}^{T}\left(f(s,\underline{Y}_s^{n-1},\underline{Z}_s^{n-1})
+h(\underline{Y}_s^n-\underline{Y}_s^{n-1},\underline{Z}_s^n-\underline{Z}_s^{n-1})\right)ds
\\ \notag &&+\int_{t}^{T}\left(k(s,\underline{Y}_s^{n-1})
+p(\underline{Y}_s^n-\underline{Y}_s^{n-1})\right)dA_s
+\int_{t}^{T}g(s,\underline{Y}_s^{n})\overleftarrow{dB}_{s}-
\sum_{i=1}^{m}\int_{t}^{T}\underline{Z}^{n(i)}_{s}dH^{(i)}_{s}.
\end{eqnarray}
For $n\geq1$ and $(\underline{Y}^{n-1},\underline{Z}^{n-1})\in{\cal E}_{m}\left(0,T \right)$, let, for any $t \in[0,T]$, $(y, z) \in \mathbb{R}\times \mathbb{R}^{ m}$,
$$\left\{
  \begin{array}{ll}
    \underline{F}_{n}(t,y,z)=f(t,\underline{Y}_t^{n-1},\underline{Z}_t^{n-1})
+h(y-\underline{Y}_t^{n-1},z-\underline{Z}_t^{n-1}), & \hbox{} \\
\\
  \underline{K}_{n}(t,y)=k(t,\underline{Y}_t^{n-1})+p(y-\underline{Y}_t^{n-1}).   & \hbox{}
  \end{array}
\right.$$
Since $h$ and $p$ are continuous, $\underline{F}_{n}$ and  $\underline{K}_{n}$ are continuous and for any $t \in[0,T]$, $(y, z) \in \mathbb{R}
\times \mathbb{R}^{ m}$,
\begin{eqnarray}\label{F}
|\underline{F}_{n}(t,y,z)|\leq \underline{\varphi}_{n}(t)+K(|y|+\|z\|) \; \; \text{and} \; \;
|\underline{K}_{n}(t,y)|\leq \underline{\psi}_{n}(t)+K|y|,
\end{eqnarray}
where,
$$\underline{\varphi}_{n}(t)=f_t+2K(|\underline{Y}_t^{n-1}|+\|\underline{Z}_t^{n-1}\|)\; \; \text{and} \; \;  \underline{\psi}_{n}(t)=k_t+2K|\underline{Y}_t^{n-1}|.$$
Consequently, since for $n\geq1$, $(\underline{Y}^{n-1},\underline{Z}^{n-1})\in{\cal E}_{m}\left(0,T \right)$ and from  {\bf (C2) (iii)}, we have for all $\mu,\,\lambda>0$ and for $n\geq1$
$$
\displaystyle{\E \left(\int_{0}^{T}e^{\mu t+\lambda A_t}\underline{\varphi}_{n}^2(t)dt+\int_{0}^{T}e^{\mu t+\lambda A_t}\underline{\psi}_{n}^2(t)dA_t\right)<\infty}.
$$
Therefore, by Theorem \ref{te}, GBDSDEL $(\xi,\underline{F}_{n},g,\underline{K}_{n},A )$, i.e., Eq. \eqref{11} has a minimal solution, which we steal denote by $(\underline{Y}^{n},\underline{Z}^{n})\in{\cal E}_{m}\left(0,T \right)$, for each $n\geq1$. To complete the proof, it suffices to show that the sequence $(\underline{Y}^{n},\underline{Z}^{n})$ converges to $(\underline{Y},\underline{Z})$ which is the minimal solution of GBDSDEL $(\xi,f,g,k,A )$, i.e., Eq. \eqref{a111}.

\smallskip\noindent{\it Step 1 : $(\underline{Y}^{n})_{n\geq0}$ is increasing with upper bound.}

\smallskip\noindent Let $(\overline{Y}^0,\overline{Z}^0)\in {\cal E}_{m}\left(0,T \right)$ be the unique solution of
\begin{eqnarray}\label{110}
\overline{Y}_t^{0}&=&\xi+\int_{t}^{T}\left(f_s+K|\overline{Y}_s^{0}|+K\|\overline{Z}_s^{0}\|\right)ds
+\int_{t}^{T}\left(k_s+K|\overline{Y}_s^{0}|\right)dA_s
+\int_{t}^{T}g(s,\overline{Y}_s^{0})\overleftarrow{dB}_{s}\notag\\ &&-
\sum_{i=1}^{m}\int_{t}^{T}\overline{Z}^{0(i)}_{s}dH^{(i)}_{s}.
\end{eqnarray}

\smallskip\noindent
Then, we have for any\ $n\geq 0$,
\begin{eqnarray}\label{cr}
\underline{Y}_t^{0}\leq \underline{Y}_t^{n}\leq \underline{Y}_t^{n+1}\leq \overline{Y}_t^{0},\, \, \, \P\mbox{-a.s}. \ \  \forall\ t\in [0,T].
\end{eqnarray}
Indeed, for $n\geq0$, let $$(Y^{n+1,n},Z^{n+1,n})=(\underline{Y}_t^{n+1}-\underline{Y}_t^{n},\underline{Z}_t^{n+1}-\underline{Z}_t^{n}).$$ 
Then, the processes $(Y^{n+1,n},Z^{n+1,n})$ satisfy the following GBDSDEL : for all $t\in[0,T]$,
\begin{eqnarray*}\label{}
Y_t^{n+1,n}&=&\int_{t}^{T}\left(h(Y_s^{n+1,n},Z_s^{n+1,n})+
\phi^n(s)\right)ds+\int_{t}^{T}\left(p(Y_s^{n+1,n})+
\psi^n(s)\right)dA_s\notag\\&&+\int_{t}^{T}g^{n}(s,Y_s^{n+1,n})\overleftarrow{dB}_{s}-
\sum_{i=1}^{m}\int_{t}^{T}Z^{(i)n+1,n}_{s}dH^{(i)}_{s},
\end{eqnarray*}
where, for all $n\geq 0$,\, the processes  $\phi^n$ and
$\psi^n$, and the function $g^n$ are given by :
\begin{eqnarray*}
g^n(s,y)&=&g(s,y+\underline{Y}_{s}^n)-g(s,\underline{Y}_{s}^n)
\end{eqnarray*}
and for $n=0$
$$\left\{
  \begin{array}{ll}
    \phi^0(s)=f(s,\underline{Y}_{s}^{0},\underline{Z}_{s}^{0})
+K(|\underline{Y}_{s}^{0}|+\|\underline{Z}_{s}^{0}\|)+f_s, & \hbox{} \\
\\
   \psi^0(s)=k(s,\underline{Y}_{s}^{0})
+K|\underline{Y}_{s}^{0}|+k_s, & \hbox{}
  \end{array}
\right.$$
for $n\geq1$
$$\left\{
  \begin{array}{ll}
    \phi^n(s)=f(s,\underline{Y}_{s}^{n},\underline{Z}_{s}^{n})-f(s,\underline{Y}_{s}^{n-1},
\underline{Z}_{s}^{n-1})-h(\underline{Y}_s^n-\underline{Y}_s^{n-1},\underline{Z}_s^n-\underline{Z}_s^{n-1}), & \hbox{} \\ \\
    \psi^n(s)=k(s,\underline{Y}_{s}^{n})-k(s,\underline{Y}_{s}^{n-1})-p(\underline{Y}_s^n-\underline{Y}_s^{n-1}). & \hbox{}
  \end{array}
\right.$$
For $n=0$, since $(\underline{Y}^{0},\underline{Z}^{0})\in \mathcal{E}_{m}(0,T)$, and from  {\bf (C2) (vi)},
it follows that,

$\phi^0(t)\geq0$, and $\psi^0(t)\geq0$, for any $t\in [0,T]$, and
${\E}\left(\displaystyle\int_{0}^{T }\left|
\phi^0(s)\right|^{2}ds+\int_{0}^{T }\left|
\psi^0(s)\right| ^{2}dA_s\right) <+\infty$. \\ Moreover, by its definition, $g^0$ satisfies {\bf (C2) (ii)} and {\bf (iv)}.
Therefore, from Lemma \ref{l1}, we get $Y_t^{1,0}\geq 0$ a.s., i.e., $\underline{Y}_t^{0}\leq \underline{Y}_t^{1}$, a.s., for all $t\in [0,T]$.

\smallskip
\noindent Now, suppose that there exists $n\geq 1$, such that $\underline{Y}_t^{n-1}\leq \underline{Y}_t^{n}$. 
For such $n$, and since $(\underline{Y}^{n-1},\underline{Z}^{n-1}), (\underline{Y}^{n},\underline{Z}^{n})\in \mathcal{E}_{m}(0,T)$, one can show, by their definitions, that, $\phi^n$, $\psi^n$ and $g^n$ satisfy all assumptions of Lemma \ref{l1}. Hence, by Lemma \ref{l1}, we have $Y^{n+1,n}\geq 0$ a.s., i.e. $\underline{Y}_t^{n}\leq \underline{Y}_t^{n+1}$, a.s., for all $t\in [0,T]$.
%

\smallskip
\noindent Now, let show that for all $n\geq 0,$ \; $\underline{Y}_t^{n}\leq \overline{Y}_t^{0}$, a.s., for all $t\in [0,T]$.

\smallskip
\noindent First, for $n=0$, let
$$(\tilde{Y}_t^{0,0},\tilde{Z}_t^{0,0})=(\overline{Y}_t^{0}-\underline{Y}_t^{0}, \overline{Z}_t^{0}-\underline{Z}_t^{0}).$$
Then, it follows from \eqref{11a} and \eqref{110} that for all $t\in[0,T]$,
\begin{eqnarray*}\label{}
\tilde{Y}_t^{0,0}&=&
\int_{t}^{T}f_s^{0}ds
+\int_{t}^{T}k_s^{0}dA_s\notag+\int_{t}^{T}g^{0}(s,\tilde{Y}_s^{0,0})\overleftarrow{dB}_{s}-
\sum_{i=1}^{m}\int_{t}^{T}\tilde{Z}^{(i)0,0}_{s}dH^{(i)}_{s},
\end{eqnarray*}
where, 
$$\left\{
   \begin{array}{ll}
     f_s^{0}=K\left(|\overline{Y}_s^{0}|+\|\overline{Z}_s^{0}\|+|\underline{Y}_s^{0}|+\|\underline{Z}_s^{0}\|\right)\geq0, & \hbox{} \\
     k_s^{0}=K\left(|\overline{Y}_s^{0}|+|\underline{Y}_s^{0}|\right)\geq0. & \hbox{}
   \end{array}
 \right.$$
\noindent 
Moreover, since, $(\overline{Y}^0,\overline{Z}^0),(\underline{Y}^0,\underline{Z}^0)\in {\cal E}_{m}\left(0,T \right)$, \;${\E}\left(\displaystyle\int_{0}^{T }\left|f_s^{0}\right| ^{2}ds+\int_{0}^{T }\left|k_s^{0}\right| ^{2}dA_s\right) <+\infty$.\\ Therefore, by Lemma \eqref{l1}, we have $\tilde{Y}_t^{0,0}\geq 0$ a.s., i.e. $\underline{Y}_t^{0}\leq \overline{Y}_t^{0}$, a.s., for all $t\in [0,T]$.
\\ Now, suppose that there exists $n\geq 0$, such that  $\underline{Y}_t^{n}\leq \overline{Y}_t^{0}$, and let us show that, $\underline{Y}_t^{n+1}\leq \overline{Y}_t^{0}$.

\noindent Since $\overline{Y}_t^{0}\geq\underline{Y}_t^{n}$, it follows from assumption $\textbf{(C2) (vi)}$ that
$$
f(t,\overline{Y}_t^0,
\overline{Z}_t^0)\geq f(t,\underline{Y}_{t}^{n},\underline{Z}_{t}^{n})+ h(\overline{Y}_t^0-\underline{Y}_t^{n},
\overline{Z}_t^0-\underline{Z}_t^{n})=\underline{F}_{n+1}(t,\overline{Y}_t^0,
\overline{Z}_t^0)
$$
and
$$
k(t,\overline{Y}_t^0)\geq k(t,\underline{Y}_{t}^{n})+ p(\overline{Y}_t^0-\underline{Y}_t^{n}) =\underline{K}_{n+1}(t,\overline{Y}_t^0).
$$
Hence, \; $f_t+ K\left(|\overline{Y}_t^{0}|+\|\overline{Z}_t^{0}\|\right)\geq \underline{F}_{n+1}(t,\overline{Y}_t^0,
\overline{Z}_t^0)\text{\; and\; } k_t+ K|\overline{Y}_t^{0}|\geq \underline{K}_{n+1}(t,\overline{Y}_t^0).$

\smallskip
\noindent
Moreover, since $\underline{F}_{n+1}(t,y,
z)=f(t,\underline{Y}_t^{n},\underline{Z}_t^{n})+ h(y-\underline{Y}_t^{n},
z-\underline{Z}_t^{n})$, and $h$ satisfies inequality \eqref{ad3}, then $\underline{F}_{n+1}$ satisfies inequality \eqref{adprop}. Consequently, from Theorem \ref{tc1}, \; $\overline{Y}_t^{0}\geq\underline{Y}_t^{n+1}$. Hence, for all $n\geq 0,$ \; $\underline{Y}_t^{n}\leq \overline{Y}_t^{0}$, a.s., for all $t\in [0,T]$.

\smallskip\noindent{\it Step 2 : A priori estimates}

\smallskip\noindent There exists a constant $C>0$ independent of $n$ such that
\begin{eqnarray}\label{esti}
&&\underset{n\geq0}{\sup}\left(\mathbb{E}\left(%
\underset{0\leq t\leq T}{\sup} |\underline{Y}_t^{n}|^{2}
+\int_{0}^{T}|\underline{Y}_s^{n}|^{2} (ds+ dA_s)+\int_{0}^{T}\|\underline{Z}_s^{n}\|^{2} ds\right)\right)
\leq C.
\end{eqnarray}
Indeed, first by virtue of \eqref{cr}, we derive that 
\begin{eqnarray}\label{A0}
&&\underset{n\geq0}{\sup}\left(\mathbb{E}\left(%
\underset{0\leq t\leq T}{\sup} |\underline{Y}_t^{n}|^{2}
+\int_{0}^{T}|\underline{Y}_s^{n}|^{2} (ds+ dA_s)\right)\right)\notag\\
&\leq&
\mathbb{E}\left(%
\underset{0\leq t\leq T}{\sup} \left(|\underline{Y}_t^{0}|^{2}+|\overline{Y}_t^{0}|^{2}\right)
+\int_{0}^{T}\left(|\underline{Y}_s^{0}|^{2}+|\overline{Y}_s^{0}|^{2} \right)(ds+ dA_s)\right).
\end{eqnarray}

\smallskip\noindent
On the other hand, applying Itô's formula to $e^{\mu t+\lambda A_t} \left|\underline{Y}_{t}^{n}\right| ^{2}$, for any $\mu,\lambda>0$, we have
\begin{eqnarray}
&&e^{\mu t+\lambda A_t}|\underline{Y}_{t}^n|^{2}+\lambda\int_{t}^{T}e^{\mu s+\lambda A_s}|\underline{Y}_{s}^{n}|^{2}dA_s
+\mu\int_{t}^{T}e^{\mu s+\lambda A_s}|\underline{Y}_{s}^{n}|^{2}ds+\int_{0}^{T
}e^{\mu s+\lambda A_s}\left\|\underline{Z}_{s}^{n}\right\| ^{2}ds\nonumber\\
&&=e^{\mu T+\lambda A_T}|\xi|^{2}
+2\int_{t}^{T}e^{\mu s+\lambda A_s}\underline{Y}_{s}^{n}\underline{F}_n(s,\underline{Y}_{s}^{n},\underline{Z}_{s}^{n})ds+
2\int_{t}^{T}e^{\mu s+\lambda A_s}\underline{Y}_{s}^{n}g(s,\underline{Y}_{s}^{n})\overleftarrow{dB}_{s} \\
&&+2\int_{t}^{T}e^{\mu s+\lambda A_s}\underline{Y}_{s}^{n}\underline{K}_n(s,\underline{Y}_{s}^{n})dA_s-
2\sum_{i=1}^{m}\int_{t}^{T}e^{\mu s+\lambda A_s}\underline{Y}_{s}^{n}\underline{Z}_{s}^{n(i)}dH_{s}^{(i)}+
\int_{t}^{T}e^{\mu s+\lambda A_s}|g(s,\underline{Y}_{s}^{n})|^{2}ds\nonumber\\&&
-\sum_{i,j=1}^{m}\int_{t}^{T}e^{\mu s+\lambda A_s}\underline{Z}_{s}^{n(i)}\underline{Z}_{s}^{n(j)}
d\left([H^{(i)},H^{(j)}]_s-\delta_{ij}s\right).\nonumber\label{est}
\end{eqnarray}
From, ${\bf (C2) (ii), (iv), (vi)}$ together with \eqref{F} and Young's inequality, we have for any $\sigma>0$ and $\gamma>0$,
\begin{eqnarray*}
&&2\left\langle \underline{Y}_{s}^{n},\underline{F}_{n}(s,\underline{Y}_s^{n},\underline{Z}_s^{n})\right\rangle\\
&\leq&2|\underline{Y}_{s}^{n}|\Big(f_s+
2K|\underline{Y}_{s}^{n-1}|+2K\|\underline{Z}_{s}^{n-1}\|+K|\underline{Y}_{s}^{n}|+K\|\underline{Z}_{s}^{n}\|\Big)
\\&\leq&
\left(1+4K+\frac{4K^{2}}{\gamma}+\frac{K^{2}}{\sigma}\right)\left|
\underline{Y}_{s}^{n}\right|^2+2K\left|
\underline{Y}_{s}^{n-1}\right|^2+\sigma\left\| \underline{Z}_{s}^{n}\right\|^2+\gamma\left\| \underline{Z}_{s}^{n-1}\right\|^2+f_s^{2},
\end{eqnarray*}
\begin{eqnarray*}
2\left\langle \underline{Y}_{s}^{n},\underline{K}_{n}(s,\underline{Y}_s^{n})\right\rangle
&\leq&2|\underline{Y}_{s}^{n}|\Big(k_s+
2K|\underline{Y}_{s}^{n-1}|+K|\underline{Y}_{s}^{n}|\Big)
\\&\leq&
\left(1+4K\right)\left|
\underline{Y}_{s}^{n}\right|^2+2K\left|
\underline{Y}_{s}^{n-1}\right|^2+k_s^{2}
\end{eqnarray*}
and
\begin{eqnarray*}
\left|g(s,\underline{Y}_s^{n})\right|^2&\leq&
K\left|
\underline{Y}_{s}^{n^{}}\right|^2.
\end{eqnarray*}
Therefore taking expectation in both side of $\eqref{est}$ with the suitable $\lambda$ and $\mu$,
and by virtue of \eqref{A0} and assumptions {\bf (C1)}, {\bf (C2) (iii)}, there exists a constant $c>0$
independent of $n$ such that for any $\gamma>0$, $\sigma>0$, we derive, for any $n\geq1$, 
\begin{eqnarray*}
{\E}\left(\int_{0}^{T
}e^{\mu s+\lambda A_s}\left\|\underline{Z}_{s}^{n}\right\| ^{2}ds\right)\notag
&\leq&c+\sigma{\E}\left(\int_{0}^{T
}e^{\mu s+\lambda A_s}\left\|\underline{Z}_{s}^{n}\right\| ^{2}ds\right)+\gamma{\E}\left(\int_{0}^{T
}e^{\mu s+\lambda A_s}\left\|\underline{Z}_{s}^{n-1}\right\| ^{2}ds\right).
\end{eqnarray*}
Consequently, choosing $\sigma<1$, we obtain, for any
$\gamma>0$, $n\geq1$,
\begin{eqnarray*}
{\E}\left(\int_{0}^{T
}e^{\mu s+\lambda A_s}\left\|\underline{Z}_{s}^{n}\right\| ^{2}ds\right)\notag
&\leq&\frac{c}{1-\sigma}+\frac{\gamma}{1-\sigma}{\E}\left(\int_{0}^{T
}e^{\mu s+\lambda A_s}\left\|\underline{Z}_{s}^{n-1}\right\| ^{2}ds\right)
\end{eqnarray*}
which, by iteration, provides, for any $n\geq1$,
\begin{eqnarray*}
{\E}\left(\int_{0}^{T
}e^{\mu s+\lambda A_s}\left\|\underline{Z}_{s}^{n}\right\| ^{2}ds\right)\notag
&\leq&\frac{c}{1-\sigma}\sum_{k=0}^{n-1}\left(\frac{\gamma}{1-\sigma}\right)^{k}
+\left(\frac{\gamma}{1-\sigma}\right)^{n}{\E}\left(\int_{0}^{T
}e^{\mu s+\lambda A_s}\left\|\underline{Z}_{s}^{0}\right\| ^{2}ds\right).
\end{eqnarray*}
Now, choosing $\gamma>0$, such that, $0<\frac{\gamma}{1-\sigma}<1$
and since $\|Z^{0}\|_{\mathcal{M}^{2}(\R^{m})}^{2}<+\infty$, 
we get the existence of a constant $C>0$ independent of $n$ such that
\begin{eqnarray*}
\underset{n\geq0}{\sup}
\left(\E\int_{0}^{T}\|\underline{Z}_s^{n}\|^2ds\right)<C.
\end{eqnarray*}
Therefore, setting $F_{s}^{n}=\underline{F}_{n}(s,\underline{Y}_{s}^{n},\underline{Z}_s^{n})$ and
$K_{s}^{n}=\underline{K}_{n}(s,\underline{Y}_{s}^{n})$,
it follows from \eqref{F} and \eqref{A0} 
that
\begin{eqnarray*}
\underset{n\geq0}{\sup}\
\left(\E\left(\int_{0}^{T}|F_{s}^{n}|^2ds+
\int_{0}^{T}|K_{s}^{n}|^2dA_s\right)\right)<+\infty.
\end{eqnarray*}

\smallskip\noindent{\it Step 3 : Convergence result}

\smallskip\noindent From \eqref{cr}  and \eqref{A0},  $(\underline{Y}^{n})_{n\geq0}$ is an increasing and bounded sequence in $\mathcal{S}^{2}(0,T)
 \cap\mathcal{A}^{2}(0,T)$. 
 Then, there exists a process $\underline{Y}$ such that $\underline{Y}_{t}^n\; \nearrow\; \underline{Y}_t\; a.s.$, for all $t\in[0, T]$. Therefore, it follows from Fatou's lemma together with the dominated convergence
theorem that
\begin{eqnarray*}\label{}
\mathbb{E}\left(%
 |\underline{Y}_t^{}|^{2}
+\int_{0}^{T}|\underline{Y}_s^{}|^{2}ds+\int_{0}^{T}|\underline{Y}_s^{}|^{2}dA_s\right)
\notag
&\leq&C
\end{eqnarray*}
and
\begin{eqnarray}\label{step2}
\displaystyle\lim_{n\rightarrow+\infty}{\E}\left(%
 |\underline{Y}_{s}^{n}-\underline{Y}_t^{}|^{2}
+\int_{0}^{T}|\underline{Y}_{s}^{n}-\underline{Y}_s^{}|^{2} ds+\int_{0}^{T}|\underline{Y}_{s}^{n}-\underline{Y}_s^{}|^{2}dA_s\right)
=0.
\end{eqnarray}
Moreover, applying again Itô's formula to $|\underline{Y}_{t}^{p}-\underline{Y}_{t}^{n}|^2$, we have
\begin{eqnarray}\label{itc}
&&\left|\underline{Y}_{t}^{p}- \underline{Y}_{t}^{n}\right| ^{2}+\int_{t}^{T
}\left\|\underline{Z}_{s}^{p}-\underline{Z}_{s}^{n}\right\| ^{2}ds
\notag\\&=&2\int_{t}^{T }\left\langle \underline{Y}_{s}^{p}-\underline{Y}_{s}^{n},F_{s}^{p}-F_{s}^{n}\right\rangle ds
+2\int_{t}^{T }\left\langle \underline{Y}_{s}^{p}-\underline{Y}_{s}^{n},K_{s}^{p}-K_{s}^{n}\right\rangle dA_s\notag\\
&&+\int_{t}^{T }\left| g(s,\underline{Y}_s^{p})-g(s,\underline{Y}_s^{n})\right|^{2}ds+M_T^{p,n}-M_t^{p,n}+N_t^{p,n}+X_T^{p,n}-X_t^{p,n},
\end{eqnarray}
where, $M$, $N$ and $X$are uniformly integrable martingale defined by :
$$\left\{
  \begin{array}{lll}
    M_t^{p,n}&=&\displaystyle
-2\sum_{i=1}^{m}\int_{0}^{t}\langle \underline{Y}_{s}^{p}-
\underline{Y}_{s}^{n},\underline{Z}_{s}^{p(i)}-\underline{Z}_{s}^{n(i)}\rangle dH_{s}^{(i)}
\\
    N_t^{p,n}&=&\displaystyle2\int_{t}^{T }\langle \underline{Y}_{s}^{p}-\underline{Y}_{s}^{n},
    g(s,\underline{Y}_s^{p})-g(s,\underline{Y}_s^{n})\rangle\overleftarrow{dB}_{s},
\\X_t^{p,n}&=&\displaystyle-\sum_{i,j=1}^{m}\int_{0}^{t}\langle
\underline{Z}_{s}^{p(i)}-\underline{Z}_{s}^{n(i)},\underline{Z}_{s}^{p(j)}-\underline{Z}_{s}^{n(j)}\rangle
d\left([H^{(i)},H^{(j)}]_s-\delta_{ij}s\right).
  \end{array}\right.
$$
Then, taking the expectation in \eqref{itc}, we deduce by virtue of Hölder's inequality and {\bf (C2) (ii), (iv)}, that
\begin{eqnarray*}
\E\int_{0}^{T
}\left\|
\underline{Z}_{s}^{p}-\underline{Z}_{s}^{n}\right\| ^{2}ds &\leq& C'\left(\E\int_{0}^{T
}\left| \underline{Y}_{s}^{p}-\underline{Y}_{s}^{n}\right| ^{2}ds\right)^{\frac{1}{2}}+C'\left(\E\int_{0}^{T
}\left| \underline{Y}_{s}^{p}-\underline{Y}_{s}^{n}\right| ^{2}dA_s\right)^{\frac{1}{2}}\\
&&+K\E\int_{0}^{T
}\left| \underline{Y}_{s}^{p}-\underline{Y}_{s}^{n}\right| ^{2}ds.\nonumber
\end{eqnarray*}
Consequently, it follows from \eqref{step2}, that $(\underline{Z}_n)_{n\geq0}$ is a Cauchy sequence in
$ \mathcal{M}^{2}(0,T,{\R}^{m})$.
Then, there exists a $\mathcal{F}_{t}-$jointly measurable process $\underline{Z}\in  \mathcal{M}^{2}(0,T,{\R}^{m})$ such that
\begin{eqnarray*}
\lim_{n\rightarrow+\infty}{\E}\int_{0}^{T }\left\|\underline{Z}_s^{n}-\underline{Z}_s^{}\right\|^2 ds=0.
\end{eqnarray*}
On the other hand, taking the supremum and the expectation in \eqref{itc}, we deduce from Burkhölder-Davis-Gundy inequality, that
\begin{eqnarray*}
&&{\E}\left(\sup_{0\leq t\leq T}\left| \underline{Y}_{t}^{p}-\underline{Y}_{t}^{n}\right| ^{2}\right)
\\&\leq& C\left(\left(\E\int_{0}^{T}|\underline{Y}_{s}^{p}-\underline{Y}_{s}^{n}|^{2}ds\right)^{\frac{1}{2}}
+\left(\E\int_{0}^{T
}\left| \underline{Y}_{s}^{p}-\underline{Y}_{s}^{n}\right| ^{2}dA_s\right)^{\frac{1}{2}}+\E\int_{0}^{T}|\underline{Y}_{s}^{p}-\underline{Y}_{s}^{n}|^{2}ds\right),
\end{eqnarray*}
from which, together with \eqref{step2}, we deduce that $\P-$almost surely, $\underline{Y}^{n}$ converges uniformly to $\underline{Y}$ which is continuous, such that\  $\displaystyle\E\left(\sup_{0\leq t\leq T}|\underline{Y}_{t}|^2\right)\leq C$. And then,
\begin{eqnarray*}\label{}
\|(\underline{Y},\underline{Z})\|_{\mathcal{E}_{m}}^{2}=\mathbb{E}\left(%
 \sup_{0\leq t\leq T}|\underline{Y}_t^{}|^{2}
+\int_{0}^{T}|\underline{Y}_s^{}|^{2}dA_s+\int_{0}^{T}\|\underline{Z}_s^{}\|^{2}ds\right)\notag&\leq&C.
\end{eqnarray*}

\smallskip\noindent{\it Step 4 : $(\underline{Y},\underline{Z})$ verifies GBDSDEL $\eqref{a111}$}

\smallskip\noindent
Since $(\underline{Y}^n,\underline{Z}^n)\rightarrow (\underline{Y},\underline{Z})$ in $ \mathcal{E}_{m}(0,T )$, along a subsequence which we still denote $(\underline{Y}^n,\underline{Z}^n)$, we get $$(\underline{Y}^n,\underline{Z}^n)\longrightarrow (\underline{Y},\underline{Z}),\; \; \; dt\otimes d\P\;\; \mbox{a.e.,}$$ and there exists, $ \mathcal{M}^{2}(0,T,{\R}^{})$ such that, for all $n\geq1$, $|\underline{Z}^n|<\Pi,\; \; \; dt\otimes d\P\;\; \mbox{a.e.}$.
\\
Therefore, from {\bf (C2) (ii), (v), (vi)}, we have
\begin{eqnarray*}
\underline{F}_n(t,\underline{Y}_t^{n},\underline{Z}_t^{n})=f(t,\underline{Y}_t^{n-1},\underline{Z}_t^{n-1})
+h(t,\underline{Y}_t^{n}-\underline{Y}_t^{n-1},\underline{Z}_t^{n}-\underline{Z}_t^{n-1})
\longrightarrow f(t,\underline{Y}_{t},\underline{Z}_{t}),\; dt\otimes d\P\; a.e.,
\end{eqnarray*}
and
\begin{eqnarray*}
\underline{K}_n(t,\underline{Y}_t^{n})=k(t,\underline{Y}_t^{n-1})
+p(\underline{Y}_t^{n}-\underline{Y}_t^{n-1})
\longrightarrow k(t,\underline{Y}_{t}),\; dt\otimes d\P\; a.e.,
\end{eqnarray*}
for all $t\in[0,T]$ as $n\rightarrow +\infty$.

\medskip\noindent
Moreover from \eqref{F} and \eqref{cr}, we have
\begin{eqnarray*}
|\underline{F}_{n}(t,\underline{Y}_{t}^{n},\underline{Z}_{t}^{n})|\leq \Psi(t),\; \; \; dt\otimes d\P\;\; \mbox{a.e.},
\end{eqnarray*}
where, $\Psi(t)=f_t+2K(|\underline{Y}_{t}^{0}|+|\overline{Y}_{t}^{0}|)+2|\Pi_{t}|$, which,
from {\bf (C2) (iii)} and Cauchy-Schwarz inequality, yields, \; $\displaystyle{\E}\int_{0}^{T} \Psi(s)ds<+\infty.$
Then, it follows by the dominated convergence theorem that
  \begin{eqnarray*}
\E\int_{t}^{T}|\underline{F}_{n}(s,\underline{Y}_s^{n},\underline{Z}_s^{n})
- f(s,\underline{Y}_{s},\underline{Z}_{s})|ds\longrightarrow 0,\; \; \; \text{as}\; \; \; n \rightarrow +\infty.
\end{eqnarray*}
Also, using the same argument, we get
 \begin{eqnarray*}
\E\int_{t}^{T}|\underline{K}_{n}(s,\underline{Y}_s^{n})
- k(s,\underline{Y}_{s})|dA_s\longrightarrow 0,\; \; \; \text{as}\; \; \; n \rightarrow +\infty.
\end{eqnarray*}
On the other hand, thanks to {\bf (C2) (ii), (iv)} and BDG inequality, we know that there exists a positive constant $C>0$ independent of $n$ such that,
\begin{eqnarray*}
\E\left(\underset{0\leq t\leq T}{\sup}\left|\int_{t}^{T}g(s,\underline{Y}_{s}^{n})\overleftarrow{dB}_s- \int_{t}^{T}g(s,\underline{Y}_{s})\overleftarrow{dB}_s \right|^{2}\right)
\leq C\E\int_{0}^{T}\left|\underline{Y}_{s}^{n}-\underline{Y}_{s}\right|^2ds+\underset{n \rightarrow
\infty}{\longrightarrow} 0
\end{eqnarray*}
and
\begin{eqnarray*}
\E\left(\underset{0\leq t\leq T}{\sup}\left|\sum_{i=1}^{m}\int_{t}^{T}\underline{Z}^{n(i)}_{s}dH^{(i)}_{s}-
\sum_{i=1}^{m}\int_{t}^{T}\underline{Z}^{(i)}_{s}dH^{(i)}_{s} \right|^{2}\right)\leq C\E\int_{0}^{T}e^{\beta A(s)}|\underline{Z}_{s}^{n}- \underline{Z}_{s}|^2ds\underset{n \rightarrow
\infty}{\longrightarrow} 0.
\end{eqnarray*}
Finally, passing to the limit on both sides of GBDSDEL $(\xi,\underline{F}_{n},g,\underline{K}_{n}, A)$ \eqref{11}, 
we get that $(\underline{Y},\underline{Z})$ is a solution of GBDSDEL$(\xi,f,g,k,A)$ \eqref{a111}.

\smallskip\noindent{\it Step 5 : $(\underline{Y},\underline{Z})$ is the minimal  solution of GBDSDEL $\eqref{a111}$.}

\smallskip\noindent
Let $(Y,Z) \in {\cal E}_{m}\left(0,T \right)$ be any solution of GBDSDEL$(\xi,f,g,k,A)$ \eqref{a111}
and consider GBDSDEL $(\xi,\underline{F}_{n},g,\underline{K}_{n}, A)$ \eqref{11}
with its minimal solution $(\underline{Y}^{n},\underline{Z}^{n}) \in {\cal E}_{m}\left(0,T \right)$, for each $n\geq0$.

\smallskip\indent
For $n=0$, let $( \tilde{Y}_{t},\tilde{Z}_{t} )=( Y_{t}-\underline{Y}_{t}^{0},Z_{t}-\underline{Z}_{t}^{0} )$. Then,
\begin{eqnarray}
\tilde{Y}_{t}&=&\int_{t}^{T}\tilde{f}(s,\tilde{Y}_{s^-},\tilde{Z}_{s})ds+\int_{t
}^{T}\tilde{k}(s,\tilde{Y}_{s^-})dA_s+\int_{t}^{T}\tilde{g}(s,\tilde{Y}_{s^-})\,d\overleftarrow{B}_{s}\nonumber\\
&&-\sum_{i=1}^{m}\int_{t}^{T}\tilde{Z}^{(i)}_{s}dH^{(i)}_{s},\,\ 0\leq t\leq
T,\label{fi}
\end{eqnarray} where, $\tilde{f}$, $\tilde{k}$ and $\tilde{g}$  are defined by :
$$\left\{
  \begin{array}{ll}
   \tilde{f}(s,y,z)=f(s,Y_{s},z+\underline{Z}_{s}^{0})
+K|y-Y_{s}|+K|Z_{s}^{0}|+f_s;  & \hbox{} \\
\\
   \tilde{k}(s,y)=k(s,Y_{s})+K|y-Y_{s}|+k_s; & \hbox{} \\
\\
   \tilde{g}(s,y)=g(s,y+\underline{Y}_{s}^{0})-g(s,\underline{Y}_{s}^{0}), & \hbox{}
  \end{array}
\right.$$

\smallskip\noindent
with $\tilde{f}(s,0,0)=f(s,Y_{s},\underline{Z}_{s}^{0})+K|Y_{s}|+K|Z_{s}^{0}|+f_s$ and
$\tilde{k}(s,0)=k(s,Y_{s})+K|Y_{s}|+k_s$.

\smallskip\noindent By their definitions, it is easy to check that,
$$|\tilde{f}(t,y,z)-\tilde{f}(t,0,0)| \leq K(|y|+\|z\|) \; \text{\; and \;} \;|\tilde{k}(t,y)-\tilde{k}(t,0)| \leq K|y|.$$
On the other hand, from {\bf (C2)(ii)}, we have
$$ f(s,Y_{s},\underline{Z}_{s}^{0})\geq -K|Y_{s}|-K|Z_{s}^{0}|-f_s \; \text{\; and \;} \; k(s,Y_{s})\geq -K|Y_{s}|-k_s,$$
this implies that $ \tilde{f}(s,0,0)\geq0$  and  $\tilde{k}(s,0)\geq0$.

\smallskip\noindent Moreover, since $f$ is Lipschitz in $z$ and satisfies inequality \eqref{ad2},
$\tilde{f}$ is a Lipschitz function in $(y,z)$ and satisfies inequality \eqref{ad3}.
Therefore, by Lemma \ref{l1} and Remark \ref{rk},
it follows that $\tilde{Y}_t\geq 0,$ i.e., $Y_t\geq \underline{Y}_t^{0}, \; \text{a.s., for all}\;  t\in [0,T].
$

\smallskip\indent
Now, suppose that there exists $n\geq 1$ such that  $Y_t\geq \underline{Y}_t^{n}$ and let prove that  $Y_t\geq \underline{Y}_t^{n+1}$. \\
From {\bf (C2) (vi)} and since $Y_t\geq \underline{Y}_t^{n}$, it follows that
$$ f(s,Y_s,Z_s)\geq f(s,\underline{Y}_s^{n},\underline{Z}_s^{n})+h(Y_s-\underline{Y}_s^{n},Z_s-\underline{Z}_s^{n})= \underline{F}_{n+1}(s,Y_{s},Z_{s})$$
and 
$$ k(s,Y_s)\geq k(s,\underline{Y}_s^{n})+
p(Y_s-\underline{Y}_s^{n})= \underline{K}_{n+1}(s,Y_{s}).$$
Since, $\underline{F}_{n+1}$  and  $\underline{K}_{n+1}$ are continuous with linear growth, 
and inequality \eqref{ad2} still holds, we get from Theorem \ref{l0a2} that,
$Y_t\geq \underline{Y}_t^{n+1}, \; \text{a.s., for all}\;  t\in [0,T].$
Consequently, for all $n\geq0$, we have\;
$Y_t\geq \underline{Y}_t^{n}, \; \text{a.s., for all}\;  t\in [0,T].$
Since $(\underline{Y}^n,\underline{Z}^n)$ converges to $(\underline{Y},\underline{Z})$, we get\;
$Y_t\geq \underline{Y}_t, \; \text{a.s., for all}\;  t\in [0,T].$
That proves that $(\underline{Y},\underline{Z})$ is the minimal solution of GBDSDEL$(\xi,f,g,k,A)$ \eqref{a111}.
\end{proof}

\begin{remark}
We can prove the maximal solution result for GBDSDEL \eqref{a111} when the coefficient $f$  and $k$ satisfy the following:
(i.e., when {\bf(v)} and {\bf (vi)} in {\bf(C2)} are replaced by the following:)
\begin{description}
\item
\textbf{(v')}\; for all $(t,z)\in[0,T]\times {\R}^{m}$, $y\mapsto f(t,y,z)$ and $y\mapsto k(t,y)$ are right-continuous, for all $\omega$ a.e.

\item\textbf{(vi')}\; there exists two functions
$h:\mathbb{R} \times \mathbb{R}^{m}\rightarrow \mathbb{R}$ and $p:\mathbb{R} \rightarrow \mathbb{R}$, satisfying, for all $y\in \mathbb{R}$, $z, z^{\prime} \in \mathbb{R}^{
 m}$  :
$$\left\{
  \begin{array}{ll}
  |h(y,z)| \leq K(|y|+\|z\|) \hbox{\, and\,}   |p(y)| \leq K|y|, & \hbox{} \\
 |h(y,z)-h(y,z^{\prime})|\leq K\|z-z^{\prime}\|, & \hbox{} \\
 y\mapsto h(y,z)\hbox{\, and\, } y\mapsto p(y) \; \hbox{are continuous,} & \hbox{}
  \end{array}
\right.$$
such that for all $y\geq y^{\prime}$, $z, z^{\prime} \in \mathbb{R}^{
 m}$, $t\in[0,T]$, we have
$$f(t,y,z)-f(t,y^{\prime},z^{\prime}) \leq
h(y-y^{\prime},z-z^{\prime})\hbox{\, and\, } k(t,y)-k(t,y^{\prime}) \leq
p(y-y^{\prime}).$$
\end{description}
\end{remark}
\noindent
To this end, we need the following approximation.

Firstly, we consider the solution $(\overline{Y}^0,\overline{Z}^0)\in {\cal E}_{m}\left(0,T \right)$  of
GBDSDEL $(\xi,\overline{F}_{0},g,\overline{K}_{0}, A)$, where, \\$\overline{F}_{0}(t,y,z)=K|y|+K|z|+f_t$  and
$\overline{K}_{0}(t,y)=K|y|+k_t$, and then define recursively a sequence
$\left\{(\overline{Y}^n,\overline{Z}^n)\right\}_{n\geq1}$
in ${\cal E}_{m}\left(0,T \right)$ by : for all $t\in[0,T]$,
%
\begin{eqnarray}\label{eqn}
\overline{Y}_t^n&=&\xi+\int_{t}^{T}\left(f(s,\overline{Y}_s^{n-1},\overline{Z}_s^{n-1})
+h(\overline{Y}_s^n-\overline{Y}_s^{n-1},\overline{Z}_s^n-\overline{Z}_s^{n-1})\right)ds
\\ \notag &&+\int_{t}^{T}\left(k(s,\overline{Y}_s^{n-1})
+p(\overline{Y}_s^n-\overline{Y}_s^{n-1})\right)dA_s
+\int_{t}^{T}g(s,\overline{Y}_s^{n})\overleftarrow{dB}_{s}-\sum_{i=1}^{m}\int_{t}^{T}\overline{Z}_s^{n(i)}dH_{s}^{i}.
\end{eqnarray}
For $n\geq1$ and $(\overline{Y}^{n-1},\overline{Z}^{n-1})\in{\cal E}_{m}\left(0,T \right)$, let, for any $t \in[0,T]$, $(y, z) \in \mathbb{R}
\times \mathbb{R}^{ m}$,
$$\overline{F}_{n}(t,y,z)=f(t,\overline{Y}_t^{n-1},\overline{Z}_t^{n-1})+h(y-\overline{Y}_t^{n-1},z-\overline{Z}_t^{n-1}).$$
$$\overline{K}_{n}(t,y)=k(t,\overline{Y}_t^{n-1})+p(y-\overline{Y}_t^{n-1}).$$
Then, using similar calculations, we deduce by Theorem \ref{te} that
GBDSDEL $(\xi,\overline{F}_{n},g,\overline{K}_{n}, A)$, i.e., Eq. \eqref{eqn} has a minimal (resp. maximal) solution in ${\cal E}_{m}\left(0,T \right)$. Here, we consider the maximal solution, which we steal denote by $(\overline{Y}^{n},\overline{Z}^{n})$, for each $n\geq1$.
By similar procedures, we get the following result:

\begin{proposition}
Assume {\bf (C1)} and {\bf (C2) (i)-(iv), (v'), (vi')}. Then
\begin{description}
  \item[(i)] for all\; $n\geq0$, \;
$\underline{Y}_t^{0}\leq \overline{Y}_t^{n+1}\leq \overline{Y}_t^{n}\leq \overline{Y}_t^{0},\, \, \, \P\mbox{-a.s}. \ \  \forall t\in [0,T],$
  \item[(ii)] the sequence $((\overline{Y}^n,\overline{Z}^n))_{n\geq0}$ converges in $ \mathcal{E}_m(0,T )$ to a limit $ (\overline{Y},\overline{Z})$, which is the maximal solution of GBDSDEL \eqref{a111}.
\end{description}

\end{proposition}

\bibliographystyle{plain}

\end{document}